\documentclass[a4paper]{article}
\usepackage{amssymb}
\usepackage{amsthm}
\usepackage{amsmath}
\usepackage{graphicx}
\theoremstyle{plain}
\newtheorem{theorem}{Theorem}
\newtheorem{lemma}{Lemma}

\newtheorem{remark}{Remark}

\newcommand{\shtm}{\,\widetilde{\triangledown}\,}
\newcommand{\trdens}[1]{\widetilde{\nabla}_{{#1}}}
 
\newcommand{\bbbn}{\mathbb{N}}
\newcommand{\bbbr}{\mathbb{R}}

\begin{document}
\title{A note on Fiedler value of classes with sublinear separators}
\author{J. Ne\v set\v ril \thanks{Supported by grant ERCCZ LL-1201 of the Czech Ministry of Education and CE-ITI of GA{\v C}R}\\
\small Computer Science Institute of Charles University (IUUK and ITI)\\
 \small    Malostransk\' e n\' am.25, 11800 Praha 1, Czech Republic\\
 \small    {\tt nesetril@kam.ms.mff.cuni.cz}
 \\
 \\
\and P. Ossona de Mendez\\
\small  Centre d'Analyse et de Math\'ematiques Sociales (CNRS, UMR 8557)\\
 \small   190-198 avenue de France, 75013 Paris, France\\
\small {\tt pom@ehess.fr}
}
\date{\today}
\maketitle

\begin{abstract}
The $n$-th Fiedler value of a class of graphs $\mathcal C$ is the maximum second eigenvalue $\lambda_2(G)$ of 
a graph $G\in\mathcal C$ with $n$ vertices. In this note we relate this value to shallow minors and, as a corollary,
we determine the right order of the $n$-th Fiedler value for some minor closed classes of graphs, including the class
of planar graphs.
\end{abstract}
\section{Introduction and Statement of Results}
The {\em Laplacian} $L(G)$ of a graph $G$ of order $n$ is the $n\times n$ matrix
with degrees on the diagonal and $-1$ for adjacent pairs of vertices (i.e.
$L(G)=D(G)-A(G)$). This matrix is real and symmetric hence has all of its $n$
eigenvalues real. As Laplacian matrices are positive
semi-definite, all the eigenavalues are non-negative. 
The all-one vector is clearly an eigenvector of this matrix,
with associated eigenvalue $0$. The second
smallest eigenvalue $\lambda_2$ of $L(G)$ is called the {\em algebraic
connectivity} of $G$ \cite{Fiedler1973}, or the {\em Fiedler value} of $G$
\cite{Spielman1996,Spielman2007284}.

Let $\mathcal C$ be a class of graphs. The {\em Fiedler value} of the class is the function
$$\lambda_{2\, {\rm max}}(\mathcal C,n)=\max_{G\in\mathcal C, |G|=n}\lambda_2(G),$$
where the maximum is taken over all graphs $G\in\mathcal C$ with $n$ vertices, i.e. with order $|G|=n$.

The Fiedler value of a graph $G$ is intensively studied and it has many applications.
For example, we shall make use of the following connection, established by Spielman and Teng,
 to embeddings of $G$ in the Euclidean space.
 
\begin{lemma}[Embedding lemma, \cite{Spielman1996,Spielman2007284}]
\label{lem:embed}
 $$\lambda_2(G)=\min\frac{\sum_{ij\in E(G)}\|\vec v_i-\vec
v_j\|^2}{\sum_{i\in V(G)}\|\vec v_i\|^2},$$
where the minimum is taken over all possible choices of 
$\vec{v_1},\dots,\vec{v_n}\in\bbbr^m$ such that $\sum_i\vec v_i=\vec 0$.
\end{lemma}

In particular, if $\phi:V(G)\rightarrow\mathbb C$ and $\sum_{u\in
V(G)}\phi(u)=0$ we have 
$$\lambda_2(G)\leq\frac{\sum_{uv\in
E(G)}|\phi(u)-\phi(v)|^2}{\sum_{u\in V(G)}|\phi(u)|^2}.$$

Barri{\`e}re {\em et al.\ }obtained the
following bound for the class of all $K_h$ minor free graphs:
$$\lambda_{2\, {\rm max}}(K_h\text{ minor free},n)\leq\begin{cases}
h-2+O(\frac{1}{\sqrt{n}})&\text{if }4\leq h\leq 9\\
\gamma h\sqrt{\log h}+O(\frac{1}{\sqrt n})&\text{otherwise}
\end{cases}
$$

In this note we determine the right order of $\lambda_{2\, {\rm max}}(\mathcal C,n)$ for the class
$\mathcal C$ of $K_h$-minor free graphs and also for the class of planar graphs (both results extend
the results of \cite{Barriere2011}). We also extend \cite{Barriere2011} to a much broader variety of classes, such as classes 
defined by forbidden subdivisions or graphs with bounded degrees. There is nothing special about minor closed classes here. 
These extensions are motivated by \cite{Sparsity}.
Particularly we prove the following:

\begin{theorem}
\label{thm:1}
For every integer $h\geq 2$ we have
$$h-2\leq \lambda_{2\, {\rm
max}}(K_h\text{ minor free},n)\leq h-2+O\left(\frac{1}{\sqrt{n}}\right)
$$
\end{theorem}

\begin{theorem}
\label{thm:2}
Let $g\in\bbbn$. Then
$$2+\Omega\left(\frac{1}{n^2}\right)\leq \lambda_{2\, {\rm max}}(\text{genus
}g,n)\leq 2+O\left(\frac{1}{\sqrt{n}}\right)$$
\end{theorem}

This will be a consequence of a more general statement using the following
definitions and notations (see \cite{ECM2009,Sparsity}).
Let $r$ be a half-integer and let $G$ be a graph. A graph $H$ is a {\em shallow
topological minor} of $G$ at {\em depth} $r$ if one can find
as a subgraph of $G$ a subdivision of $H$ where all the edges are subdivided at
most $2r$ times (we call such a subdivision a {\em $\leq 2r$-subdivision} of
$H$). The set of all the shallow topological minors of $G$ at depth $r$ is
denoted by $G\shtm r$ and we denote by $\trdens{r}$ the {\em topological grad}
of $G$ with depth $r$, which is the maximum of $\|H\|/|H|$ over all graphs $H\in G\shtm r$. 
Here, for a graph $H=(V,E)$ we put $|H|=|V|$ and $\|H\|=|E|$.
Also, for a class of graphs $\mathcal C$, we define $\mathcal C\shtm r=\bigcup_{G\in\mathcal C}G\shtm r$ 
and  $\trdens{r}(\mathcal C)=\sup_{G\in\mathcal C}\trdens{r}(G)$. 

\begin{theorem}
\label{thm:main}
Let $\mathcal C$
be a monotone class with sub-linear separators and bounded $\trdens{1/2}$. Let
$s(n)$ denote the maximum size of a vertex separator of graphs $G\in\mathcal C$ of order at most
$n$. Then
$$\lambda_{2\, {\rm max}}(\mathcal C,n)\leq
\omega\left(\mathcal C\shtm\frac{1}{2}\right)-1+O\left(\frac{s(n)}{n}\right).
$$
\end{theorem}

\section{Proofs}
In this section, we shall prove Theorem~\ref{thm:main}, then Theorem~\ref{thm:1} and Theorem~\ref{thm:2}.
Following the proof of Barri{\`e}re {\em et al.\ }\cite{Barriere2011},
we state two lemmas allowing to bound $\lambda_2(G)$ by the density of edges incident to a small
subset of vertices of $G$.
\begin{lemma}
\label{lem:fourpoints}
Let $n_{1,1},n_{1,2},n_{2,1},n_{2,2}$ be positive integers such that
\begin{align*}
n_{1,1}&\leq n_{1,2}\leq 2n_{1,1}\\
n_{2,1}&\leq n_{2,2}\leq 2n_{2,1}\\
n_{1,1}+n_{1,2}&\leq n_{2,1}+n_{2,2}\leq 2(n_{1,1}+n_{1,2})
\end{align*}
Then there exist $z_{1,1},z_{1,2},z_{2,1},z_{2,2}\in\mathbb S^2$
(where $\mathbb S^2=\{z\in\mathbb C: |z|=1\})$ such that
$$n_{1,1}z_{1,1}+n_{1,2}z_{1,2}+n_{2,1}z_{2,1}+n_{2,2}z_{2,2}=0$$
\end{lemma}
\begin{proof}
Let $n_1=n_{1,1}+n_{1,2}, n_2=n_{2,1}+n_{2,2}$ and $n=n_1+n_2$.
Define the real numbers $x_1=2/3$ and $x_2=-\frac{n_2}{n_1}z_1$, so that
$n_1x_1+n_2x_2=0$ and $-2/3\leq x_2\leq -1/3$.

For $0<x<1$ define the function $f_x:\mathbb S^2\rightarrow\mathbb S^2$ such
that $f_x(z)$ is the intersection of the unit circle and of the line through $z$
and $x$. Let $g:]0,1[\times \mathbb S^2\rightarrow\mathbb R$ be defined by
$g(x,z)=\frac{f_x(z)-x}{x-z}$. Notice that $g$ is continuous.
As $g(x_1,1)=5$ and $g(x_1,-1)=1/5$ there exists $z_{1,1}$ such that
$g(x_1,z_{1,1})=n_{1,2}/n_{1,1}$. Also, as $g(x_2,1)=1/2$ and $g(x_2,-1)=2$
there exists $z_{2,1}$ such that $g(x_2,z_{2,1})=n_{2,2}/n_{2,1}$.
Let $z_{1,2}=f_{x_1}(z_{1,1})$ and $z_{2,2}=f_{x_2}(z_{2,1})$. Then
$x_1=\frac{n_{1,1}z_{1,1}+n_{1,2}z_{1,2}}{n_1}$ and
$x_2=\frac{n_{2,1}z_{2,1}+n_{2,2}z_{2,2}}{n_2}$. Thus
$n_{1,1}z_{1,1}+n_{1,2}z_{1,2}+n_{2,1}z_{2,1}+n_{2,2}z_{2,2}=0$.
\end{proof}

\begin{lemma}
\label{lem:sep}
Let $\mathcal C$ be a monotone class of graphs and let $s(n)$
denote the maximum size of a vertex separator of a graph $G\in\mathcal C$ with
order at most $n$.

Then, for every graph $G\in\mathcal C$ with order $n$
there exists a subset $S\subset V(G)$ of cardinality at most $s(n)+2s(2n/3)$
such that:
$$\lambda_2(G)\leq \frac{e(S,V-S)}{n-|S|}.$$
\end{lemma}
\begin{proof}
Let $S_0$ be a vertex separator of $G$ of size at most $s(n)$ and let
$(Z_1,Z_2)$ be a partition of $V-S_0$ such that $|Z_1|\leq |Z_2|\leq 2|Z_1|$
and no edge exists between $Z_1$ and $Z_2$.
Let $S_1$ (resp. $S_2$) be separators of size at most $s(2n/3)$ of $G[Z_1]$
(resp. $G[S_2]$), let $(Z_{1,1},Z_{1,2})$ (resp. $(Z_{2,1},Z_{2,2})$) be a partition of 
$Z_1$ (resp. $Z_2$) such that  $|Z_{i,1}|\leq |Z_{i,2}|\leq 2|Z_{i,1}|$ and
no edge exists between $Z_{i,1}$ and $Z_{i,2}$ in $G[Z_i]-S_i$.
According to Lemma~\ref{lem:fourpoints}, there exists four complex numbers
$z_{1,1},z_{1,2},z_{2,1}$ and $z_{2,2}$ with $|z_{i,j}|=1$ and
$\sum_{i=1}^2\sum_{j=1}^2\, |Z_{i,j}|\,z_{i,j}=0$.
Define $\phi:V(G)\rightarrow\mathbb C$ as follows
$$
\phi(v)=\begin{cases}
z_{i,j}&\text{if }v\in Z_{i,j}\\
0&\text{otherwise}
\end{cases}
$$
Then we have, according to Lemma~\ref{lem:embed}:
\begin{align*}
\lambda_2(G)\leq \frac{\sum_{uv\in E(G)}|\phi(u)-\phi(v)|^2}{\sum_{u\in
V(G)}|\phi(u)|^2}=\frac{e(S,V-S)}{n-|S|},
\end{align*}
where $S=S_0\cup S_1\cup S_2$ has cardinality at most $s(n)+2s(2n/3)$.
\end{proof}

\begin{remark}
If $G$ has maximum average degree $d$ then $\lambda_2(G)\leq d(d+1)n$. 
This also follows from Lemma~\ref{lem:sep}.
Indeed,
$G$ has a proper coloration with $d+1$ colors. If $S$ is the union of the $d$
smallest color classes, then $V(G)-S$ is disconnected hence may be easily split
into four parts having approximately the same size. 
\end{remark}

We now give a general bound for the size of a bipartite subgraph of a graph
$G$ in terms of the maximum average degree and maximum clique size of shallow
topological minors of $G$.

\begin{lemma}
\label{lem:bipdens}
Let $A,B$ be disjoint vertices of a graph $G$, with $|A|\geq |B|$ and
$n=|A|+|B|$. Then 
$$
e(A,B)\leq (\omega(G\shtm 1/2)-1)\,n+(\trdens{0}(G)-\omega(G\shtm
1/2)+1)(\trdens{1/2}(G)+1)\,|B|.$$
In particular, if $G\in\mathcal C$ and $\mathcal C$ is a minor closed class with 
maximum average degree $d=2\trdens{0}(\mathcal C)$ and
clique number $\omega=\omega(\mathcal C)$, we get:
 $$
e(A,B)\leq (\omega-1)\,n+(d/2+1)(d/2+1-\omega)\,|B|.$$
\end{lemma}
\begin{proof}
Let $\omega=\omega(G\shtm
\frac{1}{2})$. Partition $A$ into $A_1$ and $A_2$ such that $A_1$ contains the
vertices with degree at most $\omega-1$ and $A_2$ contains the vertices
with degree at least $\omega$. 

Consider any linear ordering $x_1,\dots,x_p$ of $A_2$. 
We construct $H\in G\shtm\frac{1}{2}$ as follows. At the beginning, $H$ is the
empty graph with vertex set $B$. For each vertex $x_i$ in $A_2$, if $x_i$ has
two neighbours $u,v$ in $B$ that are not adjacent in $H$ we (choose one such
pair of vertices and) make them adjacent in $H$ and we continue with the next
vertex of $A_2$. If we cannot continue, this means that all the neighbours of
$x_i$ are adjacent in $H$. Then by construction we have $H\oplus K_1\in
G\shtm\frac{1}{2}$ although $\omega(H\oplus K_1)>\omega$, a contradiction. 
Hence we can continue until $A_2$ is exhausted.
Then we obtain $H\in G\shtm\frac{1}{2}$ such that $\|H\|=|A_2|$ and $|H|=|B|$.
Hence we have $|A_2|\leq \trdens{1/2}(G) |B|$ and 
$$e(A_2,B)\leq \|G[A_2\cup B]\|\leq 
\trdens{0}(G)(|A_2|+|B|).$$
As the maximum degree of vertices in $A_1$ is $\omega-1$ we have
$e(A_1,B)\leq (\omega-1)(n-|A_2|-|B|)$.
Altogether, we get
$$
e(A,B)=e(A_1,B)+e(A_2,B)\leq
(\omega-1)n+(\trdens{0}(G)-\omega+1)(\trdens{1/2}(G)+1)|B|. $$
\end{proof}

\begin{proof}[Proof of Theorem~\ref{thm:main}]
According to Lemma~\ref{lem:sep} there exists, for every graph $G\in\mathcal C$
with order $n$, a subset $S\subset V(G)$ of cardinality at most $s(n)+2s(2n/3)$
such that
$\lambda_2(G)\leq \frac{e(S,V-S)}{n-|S|}$.
According to Lemma~\ref{lem:bipdens}, we have 
$$e(V-S,S)\leq (\omega(G\shtm 1/2)-1)\,n+(\trdens{0}(G)-\omega(G\shtm
1/2)+1)(\trdens{1/2}(G)+1)\,|S|.$$
As $s(n)=o(n)$, it follows that
$$\lambda_{2\, {\rm max}}(\mathcal C,n)\leq \omega\left(\mathcal
C\shtm\frac{1}{2}\right)-1+O\left(\frac{s(n)}{n}\right).$$
\end{proof}

The lower bound of Theorem~\ref{thm:1} follows from the following easy construction.

\begin{lemma}
\label{lem:join}
Let $H_1,H_2$ be graphs and let $H_1\oplus H_2$ denote the complete join of
$H_1$ and $H_2$. Then
 $$\lambda_2(H_1\oplus H_2)=\min(\lambda_2(H_1)+|H_2|,\lambda_2(H_2)+|H_1|).$$
\end{lemma}
\begin{proof}
Let $G=H_1\oplus H_2$ be the complete join of $H_1$ and $H_2$. Then
$$
L(G)=\begin{pmatrix}
L(H_1)+|H_2|I&-J\\
-J&L(H_2)+|H_1|I
\end{pmatrix}
$$
Hence if $x_1$ is an eigenvector of
$L(H_1)$ with eigenvalue $\alpha_1$ and if  $x_2$  is an
eigenvector of $L(H_2)$ with eigenvalue $\alpha_2$, both being orthogonal to
the all-one vectors, we have:
\begin{align*}
L(G)\begin{pmatrix}
x_1\\0
\end{pmatrix}
&=\begin{pmatrix}
L(H_1)+|H_2|I&-J\\
-J&L(H_2)+|H_1|I
\end{pmatrix}\begin{pmatrix}
x_1\\0
\end{pmatrix}
=(\alpha_1+|H_2|)
\begin{pmatrix}
x_1\\0
\end{pmatrix}
\intertext{and}
L(G)\begin{pmatrix}
0\\x_2
\end{pmatrix}
&=\begin{pmatrix}
L(H_1)+|H_2|I&-J\\
-J&L(H_2)+|H_1|I
\end{pmatrix}\begin{pmatrix}
0\\x_2
\end{pmatrix}
=(\alpha_2+|H_1|)
\begin{pmatrix}
0\\x_2
\end{pmatrix}
\intertext{Moreover, if $x$ is the vector with $|H_1|$ first entries equal to
$|H_2|$ and the remaining $|H_2|$ entries equal to $-|H_1|$ we have}
L(G)x&=n x
\end{align*}
With the all-one vector, which is an eigenvector of $L(G)$ with associated
eigenvalue $0$, we have determined the full spectrum of $G$. It follows that the second
smallest eigenvalue of $G$ is 
$$\lambda_2(G)=\min(\lambda_2(H_1)+|H_2|,\lambda_2(H_2)+|H_1|).$$
\end{proof}

Hence we have, for $n>h$ (as $G\oplus K_1$ is $K_{h+1}$-minor free if $G$ is
$K_h$-minor free): $$\lambda_{2\, {\rm max}}(K_{h+1}\text{ minor free},n+1)\geq
\lambda_{2\, {\rm max}}(K_{h}\text{ minor free},n)+1$$
\begin{proof}[Proof of Theorem~\ref{thm:1}]
The upper bound comes from Theorem~\ref{thm:main}. According to
Lemma~\ref{lem:join} we have, for $n\geq h$:
 $$ \lambda_{2\, {\rm
max}}(K_{h}\text{ minor free},n)\geq\lambda_2(K_{h-2}\oplus(n-h+2)K_1)=h-2.
$$
\end{proof}

\begin{lemma}
\label{lem:bipdens2}
Let $A,B$ be disjoint vertices of a graph $G$, with $|A|\geq |B|$ and
$n=|A|+|B|$. Let $p\in\bbbn$ be such that $K_{3,p}$ is not a subgraph of
$G$. 
Then 
$$
e(A,B)\leq
2n+(\trdens{0}(G)-2)((p-1)\trdens{1/2}(G)^2+\trdens{1/2}(G)+1)|B|.
$$
\end{lemma}
\begin{proof}
 Partition $A$ into $A_1$ and $A_2$ such that $A_1$ contains the
vertices with degree at most $2$ and $A_2$ contains the vertices
with degree at least $3$. 

Consider any linear ordering $x_1,\dots,x_p$ of $A_2$.
We construct a partition $Z_0,Z_1,\dots,Z_{p-1}$ of $A_2$,
$p-1$ sets $T_1,\dots,T_{p-1}$ of triples of vertices in $B$ and a graph
 $H\in G\shtm\frac{1}{2}$ as follows. At the beginning, $H$ is the
empty graph with vertex set $B$. For each vertex $x_i$ in $A_2$, if $x_i$ has
two neighbours $u,v$ in $B$ that are not adjacent in $H$ we (choose one such
pair of vertices and) make them adjacent in $H$, put $x_i$ in $Z_0$ and 
continue with the next vertex of $A_2$. If $x_i$ has three neighbours $u,v,w$
such that $\{u,v,w\}$ is not in $Z_1$, we put $\{u,v,w\}$ in $Z_i$ and continue
with the next vertex in $A_2$. Otherwise, we try to find a triples of neighbours
of $x_i$ not in $Z_2,Z_3,\dots, Z_{p-1}$. With this construction, all the
vertices of $A_2$ are exhausted for otherwise we would exhibit a $K_{3,p}$
subgraph of $G$.

Then we obtain $H\in G\shtm\frac{1}{2}$ such that $\|H\|=|Z_0|$ and $|H|=|B|$.
Hence we have $|Z_0|\leq \trdens{1/2}(G) |B|$ and the number of triangles in $H$
is at most $2\trdens{0}(H)^2\,|H|\leq 2\trdens{1/2}(G)^2\,|B|$. Each of the sets
$Z_1,\dots,Z_{p-1}$ contains only triples corresponding to triangles of $H$.
Hence for $1\leq i\leq p-1$ we have $|Z_i|\leq 2\trdens{1/2}(G)^2\,|B|$.
Altogether, we get
$$|A_2|\leq \trdens{1/2}(G)(1+(p-1)\trdens{1/2}(G)) |B|.$$
Moreover
$$e(A_2,B)\leq \|G[A_2\cup B]\|\leq 
\trdens{0}(G)(|A_2|+|B|).$$
As the maximum degree of vertices in $A_1$ is $2$ we have
$e(A_1,B)\leq 2(n-|A_2|-|B|)$.
As $e(A,B)=e(A_1,B)+e(A_2,B)$, we get
$$
e(A,B)\leq
2n+(\trdens{0}(G)-2)((p-1)\trdens{1/2}(G)^2+\trdens{1/2}(G)+1)|B|.
$$
\end{proof}
We deduce the Theorem~\ref{thm:2}, which is an extension of the inequalities obtained by
Barri{\`e}re {\em et al.\ }for planar graphs \cite{Barriere2011}.
\begin{proof}[Proof of Theorem~\ref{thm:2}]
According to Lemma~\ref{lem:sep} there exists, for every graph $G$ of genus $g$
with order $n$, a subset $S\subset V(G)$ of cardinality at most
$s(n)+2s(2n/3)=O(\sqrt{n})$ such that
$\lambda_2(G)\leq \frac{e(S,V-S)}{n-|S|}$.
As $G$ has genus $g$, it does not contain $K_{3,4g+3}$ as a
subgraph. Hence, according to Lemma~\ref{lem:bipdens2}, we have 
$e(V-S,S)\leq 2n+O(\sqrt{n})$.
It follows that
$$\lambda_{2\, {\rm max}}(\text{genus }g,n)\leq
2+O\left(\frac{1}{\sqrt{n}}\right).$$

For the lower bound, consider the planar graph $K_2\oplus P_{n-2}$, for which
$\lambda_2=4-2\cos\bigl(\frac{\pi}{n-1}\bigr)=2+\Theta\bigl(\frac{1}{n^2})$.
\end{proof}
\section{Concluding Remarks}

Remark that the same kind of argument as in the proof of Theorem~\ref{thm:2}
could be applied to prove that graphs that do not
contain $K_{p,q}$ for some $p\leq q$ but have bounded $\trdens{1/2}$ and
sub-linear separators actually have $\lambda_2$ bounded by $p-1+o(1)$ (as
$n\rightarrow\infty$).

An interesting problem is to characterize properties of a class by means of
separating and spectral properties, see e.g. Problem 16.2 of \cite{Sparsity}.
In a sense, this note may be seen as a step in this direction.
 
 \providecommand{\noopsort}[1]{}\providecommand{\noopsort}[1]{}

\end{document}